\documentclass[10pt,oneside,leqno]{amsart}

\usepackage{amsxtra}
\usepackage{amsopn}
\usepackage{amsmath,amsthm,amssymb}
\usepackage{amscd}
\usepackage{color}
\usepackage{amsfonts}
\usepackage{latexsym}
\usepackage{verbatim}

\theoremstyle{plain}
\newtheorem{theorem}{Theorem}[section]
\newtheorem{lemma}[theorem]{Lemma}
\newtheorem{prop}[theorem]{Proposition}
\newtheorem{cor}[theorem]{Corollary}

\theoremstyle{definition}
\newtheorem{definition}[theorem]{Definition}
\newtheorem{rem}[theorem]{Remark}
\newtheorem{ex}[theorem]{Example}

\newcommand\C{{\mathbb C}}

\newcommand\de{{\partial}}
\newcommand\debar{{\bar\partial}}

\newcommand\J{{\mathcal J}}
\newcommand\gen{{TM\oplus T^*M}}
\renewcommand\b{{\mathfrak b}}

\begin{document}

 \title[Tamed Symplectic forms]{Tamed Symplectic forms and Generalized Geometry}
\author{Nicola  Enrietti, Anna Fino and Gueo Grantcharov}
\date{\today}
\address{Dipartimento di Matematica \\ Universit\`a di Torino\\
Via Carlo Alberto 10\\
10123 Torino\\ Italy} \email{nicola.enrietti@unito.it, annamaria.fino@unito.it}
\address{Department of Mathematics\\
Florida International University\\
University Park\\
Miami, FL 33199\\
USA}
\email{grantchg@fiu.edu }
\thanks{This work was supported by the Projects MIUR ``Riemannian Metrics and Differentiable Manifolds'',
``Geometric Properties of Real and Complex Manifolds'' and by GNSAGA
of INdAM}
\begin{abstract}  We show that  symplectic  forms  taming complex structures on compact manifolds  are related to special types of almost generalized K\"ahler structures. By considering the commutator $Q$ of the two associated almost complex structures $J_{\pm}$,  we prove that if   either the manifold is $4$-dimensional or the distribution ${\mbox {Im}}  \, Q$ is involutive,   then the manifold can be expressed locally as a disjoint union of twisted Poisson leaves.
 \end{abstract}
\maketitle
\section{Introduction}

Let $(M,\Omega)$ be a compact  symplectic manifold of dimension $2n$. An almost complex structure $J$ on  $M$ is said to be \emph{tamed} by $\Omega$ if
$$
\Omega(JX,X)>0
$$
for any non-zero vector field $X$ on $M$.  When $J$ is a complex structure (i.e. $J$ is integrable) and $\Omega$ is tamed by $J$, the pair $(\Omega, J)$ has been also called a {\em Hermitian-symplectic structure} in \cite{ST}.
Although any symplectic structure always admits tamed almost complex structures,
it is still an open problem to find an example of a compact complex manifold admitting a
Hermitian-symplectic structure, but no K\"ahler structures. From \cite{ST,LiZhang} there exist no examples in dimension $4$.
Moreover, the study of    tamed  symplectic structures in dimension $4$ is related to a more general conjecture of Donaldson (see for instance  \cite{donaldson,weinkove,LiZhang}).

In  \cite{EFV}  it was shown that symplectic  forms  taming complex structures on compact manifolds
are strictly related to Hermitian metrics having the fundamental form $\partial \overline \partial $-closed (called  strong K\"ahler with torsion or simply ${\rm SKT}$ metrics) and some negative results on  compact quotients of Lie groups by discrete subgroups were also given. In particular,   if $M$ is   a nilmanifold  (not a torus)  endowed with an invariant complex structure $J$, then   $(M, J)$  does not admit any symplectic
form taming $J$ (\cite{EFV}).

The generalized complex structures  (\cite{gual-th},\cite{HitCY})  contain, as particular cases, the
complex and symplectic structures, so it is a natural  question to see if there is any geometrical interpretation of the tamed condition in terms of generalized complex structures.
There is an associated notion of almost generalized K\"ahler structure which consists of a pair of commuting almost generalized complex structures $(\J_1,\J_2)$ such that $\mathcal G = - {\mathcal J}_1{\mathcal J}_2$ is a positive definite metric on the bundle  $\gen$.  By \cite{gual-th} an almost generalized K\"ahler structure is equivalent to an almost bihermitian structure $(g,J_+,J_-)$ together with a 2-form $b$. The almost generalized K\"ahler structure $(\J_1,\J_2)$ is integrable if and only if $J_+,J_-$ are integrable and there is a 2-form $b$ such that
\[
J_+d\omega_+ = -J_-d\omega_- =  db,
\]
where $\omega_{\pm} ( \, \cdot, \, \cdot ) = g ( \cdot, J_{\pm} \cdot)$ denotes the fundamental form associated to the almost Hermitian structure $(J_{\pm}, g)$ and we use the convention
$$
J_{\pm} d \omega_{\pm} ( \, \cdot, \, \cdot , \, \cdot) = d \omega (  J_{\pm}  \, \cdot,  J_{\pm} \, \cdot ,  J_{\pm} \, \cdot).
$$

This bihermitian structure appeared in the physics literature as long ago as 1984  \cite{GHR} as a target space for the supersymmetric  $\sigma$-model.

In this paper we consider an almost bihermitian structure defined naturally by tamed symplectic form. Let $(M,J)$ be a complex manifold and let $\Omega$ be a symplectic form taming $J$. Then $J_-=- \Omega^{-1} J^* \Omega$ defines an almost complex structure, $g=\Omega^{1,1}J$ a Riemannian metric  compatible with $J$ and $b=-\Omega^{2,0+0,2}J$ a 2-form. In Section \ref{sec:genSym} we will prove that the bihermitian structure $(g,J_+=J,J_-,b)$ is equivalent to an almost generalized K\"ahler structure $(\J_1,\J_2)$ such that $\J_1=\J_\Omega$, i.e. is the generalized complex structure induced by the symplectic form $\Omega$ (a similar result was proven in \cite{OP} when $J_-$ is integrable). An almost generalized K\"ahler structure defined in this way will be call \emph{tamed}.

Moreover, we will determine an explicit  relation between the torsion $3$-form of the Bismut connection associated to $(J_+, g)$ and  the $3$-form $J_- d \omega_-$, where $\omega_-$ denotes the  fundamental $2$-form associated to the almost Hermitian structure $(J_-, g)$.

The second part of the paper is devoted to the study of Poisson and twisted Poisson structures on tamed almost generalized K\"ahler structure. By \cite{Hitchin} and \cite{gual-th} that there are strong relations between generalized K\"ahler and Poisson structures. Indeed, given a generalized K\"ahler structure $(J_{\pm}, g, b)$  if we consider the $(2, 0)+(0, 2)$-form with respect to both $J_{\pm}$ given by
$$
S(X, Y ) = g( Q X, Y ),
$$
where $Q=[J_+,J_-]$, then its $(0, 2)$ part can be identified with a  holomorphic Poisson structure   for either complex structure $J_+$ or $J_-$ \cite{Hitchin}. Twisted Poisson manifolds considered by Severa and Weinstein \cite{SW}  are a natural generalization
of the  Poisson structures  which also arises in string theory.  A pair  $(\pi, \phi)$, where $\phi$ is a bivector field and $\phi$ is a closed 3-form, is called a \emph{twisted Poisson structure} if it satisfies the equation
$$
[\pi, \pi]_s = \frac{1}{2} \pi^{\#} \phi,
$$
where $[\, , \, ]_s$ denotes the Schouten-Nijenhuis bracket and $\pi^{\#}$  is the vector bundle homomorphism $T^* M \rightarrow  TM$ induced by $\pi$. In the case $\phi =0$ one recovers the
usual notions of Poisson tensor and Poisson manifold.

Let  $(M,g)$ be  a Riemannian manifold and $Q:TM\to TM$ a skew-symmetric endomorphism of the tangent bundle of $M$, i.e.  such that $g(QX,Y)=-g(X,QY)$, for every   $X$ and $Y$. In Section \ref{sec:twistedRiemann}  we show that  if  $Im (Q)$ is an involutive  distribution, then it is integrable to a generalized foliation, i.e. $M$  can be
expressed locally as a disjoint union of embedded submanifolds (the leaves) such that at any point
$p \in M$, the tangent space to the leaf through $p$ is precisely $Im(Q) (p)$.  Moreover, every leaf has  a twisted Poisson structure. If additionally the rank of $Q$ is constant, then $M$ itself carries a twisted Poisson structure.

In Section \ref{sec:twistedGK} we apply the previous results with $Q=[J_+,J_-]$ to investigate the existence of twisted Poisson structures on a tamed almost generalized K\"ahler structure when $J_-$ is not integrable. We prove in particular
  that if either the manifold is $4$-dimensional or the distribution $Im(Q)$ is involutive, then $Im(Q)$ induce a foliation on the $M$ such that on every leaf $L$ we have two twisted Poisson manifold induced by the projections on $L$ of the endomorphisms $Q$ and $J_+-J_-$.  Moreover, it turns out that   the $(2,0)$-part  $\tilde{Q}^{(2,0)}$ of the  associated bivector  $\tilde Q$  is holomorphic if and only if   the $(3,0)$ part  of $\Omega(X,N_-(Y,Z))$ vanishes. So in particular for  a complex surface  $(M, J)$ endowed   with a symplectic form $\Omega$  that tames $J$,  the $(2,0)$-part  $\tilde{Q}^{(2,0)}$ of the bivector  $\tilde Q$  is holomorphic.

\vspace{5mm}

\section{Generalized K\"ahler structures and symplectic structures}\label{sec:genSym}

Let $(M,J,g)$ be an Hermitian manifold and let  $\omega ( \, \cdot, \, \cdot ) = g ( \cdot, J  \cdot)$ be its fundamental 2-form. We say that the metric $g$ is \emph{strong K\"ahler with torsion} (or simply \emph{SKT}) if the fundamental 2-form $\omega$ is $\de\debar$-closed, i.e. $\de\debar\omega=0$. The study of SKT metrics is strictly related with Hermitian connections. Any Hermitian manifold $(M,J,g)$ admits a unique linear connection $\nabla^B$, called \emph{Bismut connection}, such that $\nabla^Bg=\nabla^BJ=0$ and the torsion 3-form $c(X,Y,Z)=g(X,T^B(Y,Z))$ is skew-symmetric \cite{Bismut, Gauduchon}. This connection is also called K\"ahler with torsion, or simply KT, and we say that it is \emph{strong} if $dc=0$. It is easy to prove that the torsion 3-form $c$ of $\nabla^B$ is a multiple of $Jd\omega$, so $dc=0$ if and only if  the $J$-Hermitian metric $g$ is SKT.

SKT structures are deeply connected with symplectic forms taming complex structures. We say that a symplectic form $\Omega$ on a complex manifold $(M,J)$ \emph{tames} $J$ if it satisfies
\[
 \Omega(JX,X)>0
\]
for every non-zero vector field $X$ on $M$. In \cite{EFV} it was proved that a symplectic form taming $J$ is equivalent to an SKT metric such that $\de\omega=\debar\beta$ for a $\de$-closed $(2,0)$-form $\beta$. In particular, this implies that the torsion 3-form of the Bismut connection $c$ is exact, since $c=d(\beta+\overline\beta)$.

SKT metrics with exact torsion 3-form have also been studied because of their relations with generalized complex and K\"ahler geometry \cite{gual-th}. A \emph{almost generalized complex structure} on $M$ is an endomorphism $\J$ of $\gen$ that satisfies  the condition $\J^2=-id$ and is orthogonal with respect to the natural inner product of $\gen$. The $+i$-eigenbundle $L$ of $\J$ is a maximal isotropic subbundle of $(\gen)\otimes \C$: if $L$ is closed under the Courant bracket, we say that $\J$ is \emph{integrable}.

A pair $(\J_1,\J_2)$ of commuting generalized (almost) complex structures such that $\mathcal G = - {\mathcal J}_1{\mathcal J}_2$ is a positive definite metric on the bundle  $\gen$ is called \emph{generalized (almost) K\" ahler structure}.

Any almost generalized K\"ahler structure induces a quadruple $(g, b, J_+, J_-)$, where $g$ is a Riemannian metric, $b$ is a $2$-form and $J_{\pm}$ are almost Hermitian structures on $(M, g)$. Conversely, for any quadruple $(g, b, J_+, J_-)$ we can define an almost generalized K\"ahler structure by
\begin{equation} \label{expressionJAB}
\begin{array}{l}
{\mathcal J}_1= \frac 12  \left(  \begin{array}{cc} 1&0\\ -b & 1\end{array} \right)  \left( \begin{array}{cc
} J_- - J_+ & \omega_-^{-1} + \omega_+^{-1}\\ -( \omega_- + \omega_+)& - (J_-^* -J_+^*) \end{array}  \right)   \left(  \begin{array}{cc} 1&0\\ b & 1\end{array} \right)\\[12 pt]
{\mathcal J}_2= \frac 12   \left(  \begin{array}{cc} 1&0\\-b
 & 1\end{array} \right) \left( \begin{array}{cc
} J_- + J_+ & \omega_-^{-1} - \omega_+^{-1}\\ -( \omega_- - \omega_+)& - (J_-^* + J_+^*) \end{array}  \right)   \left(  \begin{array}{cc} 1&0\\ b & 1\end{array} \right).\\
\end{array}
\end{equation}
Note that $\J_1,\J_2$ are integrable if and only if
\begin{itemize}
\item[i)] $J_+,J_-$ are integrable;
\item[ii)] $J_+d\omega_+=-J_-d\omega_-=  db$, i.e. $g$ is SKT with respect to $J_\pm$ and both the torsion 3-forms of the Bismut connections $\nabla^B_\pm$ are exact.
\end{itemize}
Note that a symplectic form $\Omega$ on a manifold $M$ induces the generalized complex structure
\[
\mathcal J_\Omega= \left(
\begin{array}{cc}
0 &\Omega^{-1} \\
-\Omega & 0
\end{array} \right),
\]
which is integrable since $d\Omega=0$. The following proposition allows us to see symplectic structures taming complex structures as particular cases of (almost) generalized K\"ahler structures.

\begin{prop}\label{main}
Let $(M,\Omega)$ be a symplectic manifold. Giving a almost generalized K\"ahler structure on $M$ such that one of the generalized complex structures is induced by $\Omega$ is equivalent to assigning an almost complex structure $J$ such that $\Omega$ tames $J$, i.e. $\Omega(JX,X)>0$ for any vector field $X \neq 0$.

Moreover, the almost generalized K\"ahler structure is integrable if and only if the almost complex structures $J$ and $-\Omega^{-1}J_+^*\Omega$ are integrable.
\end{prop}
\begin{proof}
We consider an almost generalized K\"ahler structure in the form $(\J_1=\J_\Omega,\J_2)$, and by imposing that $\mathcal J_1=\mathcal J_\Omega$ in equations \eqref{expressionJAB} we find the conditions on the bihermitian structure $(g,b,J_\pm)$. The $(1,2)$ element of the matrix representing ${\mathcal J}_1$ gives $-\frac12(\omega_-^{-1} + \omega_+^{-1})=\Omega^{-1}$, i.e.
\begin{equation}\label{eq:metric}
 g=\frac12 \Omega(J_++J_-).
\end{equation}
Moreover, by applying this condition to the $(1,1)$ component we obtain
\begin{equation}\label{eq:form}
 b= - \frac12 \Omega(J_+ - J_-).
\end{equation}
Note that $g  \mp b  =\Omega J_\pm$, so by using that $g$ is symmetric and $b$ skew-symmetric
\begin{equation}\label{eq:J-}
J_-=-\Omega^{-1}J_+^*\Omega.
\end{equation}
Now we can prove that $\Omega$ tames $J_+$, since
\[
 \Omega J_+ = \Omega^{1,1} J_+ = \frac12(\Omega J_+-J_+^*\Omega) = g > 0.
\]
We still have to prove the integrability conditions: an almost generalized K\"ahler structure is integrable if and only if $J_+,J_-$ are integrable and $J_+d\omega_+=-J_-d\omega_-=db$. However, $b= \mp \Omega_\pm^{2,0+0,2}J_\pm$, i.e.
\[
 b_\pm^{2,0}= \mp   i\Omega_\pm^{2,0} \quad \text{and} \quad b_\pm^{0,2}=\pm  i\Omega_\pm^{0,2},
\]
where $b^{2,0}_+$ is the $(2,0)$-part of $b$ with respect to $J_+$. Then if $J_+,J_-$ are integrable the condition $J_+d\omega_+=-J_-d\omega_-=  db$ is necessarily satisfied.
\end{proof}
The bihermitian structure $(g,J_\pm,b)$ associated to $(\J_\Omega,\J_2)$ was also considered in \cite{OP}.

\begin{rem}
By a direct computation,  if ${\mathcal J}_1 = {\mathcal J}_{\Omega}$ then the second generalized complex structure $\mathcal J_2$ can be written as
\[
\mathcal J_2= \left(
\begin{array}{cc}
- 2(J_++J_-)^{-1} & -\Omega^{-1}(J_+^*-J_-^*)(J_+^*+J_-^*)^{-1} \\
-\Omega(J_+-J_-)(J_++J_-)^{-1} & +2(J_+^*+J_-^*)^{-1}
\end{array} \right).
\]
Moreover, by applying Equations (6.4) and (6.5) of \cite{gual-th} we have that
\[
\begin{split}
&L^+ = L \cap L_\Omega = \{ X +i\Omega X\ |\ X\in T_\C^{1,0}\}, \\
&L^- = L \cap \overline{L}_\Omega = \{ X-i\Omega X\ |\ \Omega X\in (T_\C^*)^{1,0}\}.
\end{split}
\]
where $L_\Omega,L$ denote the maximal isotropic subspaces associated to $\J_\Omega,\J_2$.
\end{rem}

 When both $J_+$ and $J_-$ are integrable, the condition $J_\pm d\omega_{\pm}=  \pm  db$ holds. In the following proposition we measure the failure of this condition in the non integrable case.
\begin{prop}\label{relation dc-N}
 If $(J_\pm,g,b)$ is the almost bi-Hermitian structure induced by an almost complex  structure $J$ taming a symplectic form $\Omega$, then
\[
\begin{split}
- J_+d\omega_+(X,Y,Z)  +  db(X,Y,Z) = \frac12 \sigma_{XYZ}[\Omega(J_-X,N_+(Y,Z))], \\
- J_-d\omega_-(X,Y,Z)  -  db(X,Y,Z) = \frac12 \sigma_{XYZ}[\Omega(J_+X,N_-(Y,Z))],
\end{split}
\]
where by $N_{\pm}$ we denote the  Nijenhuis tensor of $J_{\pm}$ given by
$$
N_{\pm} (X, Y) = [X, Y] - [J_{\pm} X, J_{\pm} Y] +  J_{\pm}  [J_{\pm} X, Y] +  J_{\pm}  [ X,  J_{\pm} Y]
$$
\end{prop}
\begin{proof}
We reduce to the proof of the second condition (the first is similar). Using the relations  $$g=\frac12\Omega(J_++J_-), \quad d\Omega=0, \quad \Omega J_\pm = - J_\mp^*\Omega,$$ we obtain
\[
\begin{split}
 J_-  d\omega_-&\,(X,Y,Z) = d\omega_-(J_-X,J_-Y,J_-Z) \\
=&\,\sigma_{XYZ}[J_-X\,\omega_-(J_-Y,J_-Z)+\omega_-(J_-X,[J_-Y,J_-Z])] \\
=&\, \frac12 \sigma_{XYZ}[J_-X\, \Omega((J_++J_-)Y,J_-Z)+\Omega((J_++J_-)X,[J_-Y,J_-Z])] \\
=&\, \frac12 \sigma_{XYZ}[J_-X\, \Omega(J_-Y,J_-Z)+\Omega(J_-X,[J_-Y,J_-Z])+J_-X\, \Omega(J_+Y,J_-Z)+\Omega(J_+X,[J_-Y,J_-Z])] \\
=&\, \frac12 \sigma_{XYZ}[J_-X\, \Omega(Y,Z)+\Omega(J_+X,[J_-Y,J_-Z])].
\end{split}
\]
Moreover,
\[
 \begin{split}
  db(X,Y,Z)=&\, \sigma_{XYZ}[X\, b(Y,Z)+b(X,[Y,Z])] \\
= & - \, \frac12 \sigma_{XYZ}[X\, \Omega((J_+-J_-)Y,Z)+\Omega((J_+-J_-)X,[Y,Z])] \\
= &- \, \frac12 \sigma_{XYZ}[X\, \Omega(J_+Y,Z)+\Omega(J_+X,[Y,Z])-X\, \Omega(J_-Y,Z)-\Omega(J_-X,[Y,Z])] \\
= & - \, \frac12 \sigma_{XYZ}[-X\, \Omega(Y,J_-Z)+\Omega(J_+X,[Y,Z])-X\, \Omega(J_-Y,Z)-\Omega(J_-X,[Y,Z])].
 \end{split}
\]
So
\[
\begin{split}
- J_-d\omega_-(X,Y,Z)   -db (X,Y,Z)= \frac12 \sigma_{XYZ}[-J_-X\, \Omega(Y,Z)-\Omega(J_+X,[J_-Y,J_-Z])-X\, \Omega(Y,J_-Z)\\
 +\Omega(J_+X,[Y,Z])-X\, \Omega(J_-Y,Z)-\Omega(J_-X,[Y,Z])]
\end{split}
\]
Note that since $d\Omega=0$
\[
\begin{split}
\sigma_{XYZ}[-J_-X\, \Omega(Y,Z)- &\, X\, \Omega(Y,J_-Z)-X\, \Omega(J_-Y,Z)-\Omega(J_-X,[Y,Z])]= \\
=&\, \sigma_{XYZ}[\Omega(X,[Y,J_-Z])+\Omega(X,[J_-Y,Z])  ] \\
=&\, \sigma_{XYZ}[\Omega(J_+X,J_-[Y,J_-Z]+J_-[J_-Y,Z]) ].
\end{split}
\]
Then
\[
- J_-d\omega_-(X,Y,Z) -db (X,Y,Z)=\frac12\sigma_{XYZ}\,\Omega(J_+X,[Y,Z]-[J_-Y,J_-Z]+J_-[Y,J_-Z]+J_-[J_-Y,Z]),
\]
that concludes the proof.
\end{proof}

In the sequel, we will focus our attention on a particular case.
\begin{definition}
Let $M$ be a $2n$-dimensional manifold. An almost generalized K\"ahler structure on $M$  is called \emph{tamed} if it is induced as in Proposition \ref{main} by an integrable complex structure $J$ and a symplectic form $\Omega$ taming $J$.
\end{definition}
Applying Proposition \ref{relation dc-N} we have immediately that a tamed almost generalized K\"ahler structure satisfies
\begin{equation}\label{eq:J+J-tamed}
J_+ d\omega_+=   db, \quad   (J_+d\omega_+  + J_-d\omega_-) (X,Y,Z) = -\frac12 \sigma_{XYZ}[\Omega(J_+X,N_-(Y,Z))].
\end{equation}

The following example shows that the almost complex structure $J_-=-\Omega^{-1}J^*\Omega$ is not always integrable, even in the compact four-dimensional case.

\begin{ex}\label{ESEMPIOdim4} {\bf (A  compact $4$-dimensional example)} Consider the hyper-elliptic surface that can be viewed as a compact quotient by  a lattice  of the solvable  Lie group with structure equations
\[
\begin{cases}
de^1=e^{24} \\
de^2=-e^{14},\\
de^3=de^4=0 \\
\end{cases}
\]
endowed  with the  invariant complex structure $J_+$ defined by $J_+ e_1=-e_2,\ J_+ e_3=-e_4$ and the invariant symplectic form $\Omega=e^{12}+e^{24}+e^{34}$. Note that $J_+$ is integrable and $\Omega$ tames $J_+$, but the almost complex structure $J_-=-\Omega^{-1} J_+^* \Omega$ defined by
\[
{J_-(e_1)=-e_2 + e_3 \quad J_-(e_2)=-e_4 \quad J_-(e_3)=-e_1 - e_4 \quad J_-(e_4)=e_2}
\]
is not integrable.
\end{ex}

By \cite[Theorem 3.2]{cav-form}, for every generalized K\"ahler structure $(\mathcal J_1,\mathcal J_2)$ the DGAs $(\Omega^*(\overline L_i),d_{L_i})$ are formal for $i=1,2$. Moreover, when $\mathcal J_1$ is induced by a symplectic structure, $(\Omega^*(\overline L_1),d_{L_1})$ is isomorphic to $(\Omega^*(M),d)$.

\begin{cor}
Let $(M,J)$ be a complex manifold. If it admits a symplectic form $\Omega$ taming $J$ and such that the almost complex structure $-\Omega^{-1}J^*\Omega$ is integrable, then $M$ is formal.
\end{cor}

\vspace{5mm}

\section{Twisted Poisson structures on Riemannian manifolds}\label{sec:twistedRiemann}

We recall the following
\begin{definition}
A bivector field $\pi$ on a manifold $M$ is called a \emph{twisted Poisson structure} if there exist a closed 3-form $\phi$ such that
$$
[\pi, \pi]_s = \frac {1}{2} \Lambda^3 \pi^{\#} \phi,
$$
where $\pi^{\#}: T^* M \rightarrow TM$  is the vector bundle homomorphism induced by $\pi$ and $[\,,]_s$ is the Schouten bracket on bivector fields.

\end{definition}

Let $(M,g)$ be a Riemannian manifold and $Q:TM\to TM$ a skew-symmetric endomorphism of the tangent bundle of $M$, i.e. such that  $g(QX,Y)=-g(X,QY)$, for every $X$ and $Y$. Denote with $S$ the 2-form corresponding to $Q$ as $S(X,Y)=g(QX,Y)$, and define the bivector field $\tilde Q$ as $\tilde Q=\#_2^{-1}S$, where $\#:TM\to T^*M$ is the isomorphism induced by the metric $g$. Our aim is to understand under which conditions $\tilde Q$ is a twisted Poisson structure.

\begin{lemma}\label{lemma}
Let $(M,g)$ be a Riemannian manifold and $Q:TM\to TM$ a skew-symmetric endomorphism of the tangent bundle of $M$. Then the  bivector $\tilde Q$ satisfies
$$
[\tilde Q,\tilde Q]=-2\,\#^{-1}_3\Big(\sigma_{XYZ}\,g\big((\nabla^{LC}_{QX}Q)Y,Z\big)\Big),
$$
\end{lemma}

\begin{proof}
First, note that $\nabla^{LC}_XS(Y,Z) = g((\nabla^{LC}_XQ)Y,Z)$. Using \cite{Vaisman}  $[\tilde Q,\tilde Q] = \#^{-1}_3\{2\delta_gS\wedge S-\delta_g(S\wedge S)\}$, where $\delta_g$ is the co-differential $\delta_g=-*d*$ with respect to $g$. For $\delta_g$ we have the following formula:
$\delta_g \alpha(X_1,...X_{n-1})=-\sum_i(\nabla^{LC}_{E_i}\alpha)(E_i, X_1,...,X_{n-1})$ for any $n$-form $\alpha$ and orthonormal frame $(E_1,...,E_{dim(M)})$. Then a direct calculation shows that:
\[
\begin{split}
-2\delta_gS\wedge S(X,Y,Z)=&\, 2\sum_i \Big\{g((\nabla^{LC}_{E_i} Q)E_i, X)g(QY,Z)+g((\nabla^{LC}_{E_i} Q)E_i,Y)g(QZ,X) \\
&\, +g((\nabla^{LC}_{E_i} Q)E_i,Z)g(QX,Y)\Big\}
\end{split}
\]
Also
$$
-\delta_g(S\wedge S)(X,Y,Z)=\sum_i g((\nabla^{LC}_{E_i} (S\wedge S))(E_i,X,Y,Z) = 2\sum_i((\nabla^{LC}_{E_i}) \wedge S)(E_i,X,Y,Z)
$$
When expanded the last term becomes:
\[
\begin{split}
2\Big\{&\,g((\nabla^{LC}_{E_i} Q)E_i,X)g(QY,Z)+g((\nabla^{LC}_{E_i} Q)E_i,Y)g(QZ,X)+g((\nabla^{LC}_{E_i} Q)E_i,Z)g(QX,Y) \\
&\,+g((\nabla^{LC}_{E_i} Q)Y,Z)g(QE_i,X)+g((\nabla^{LC}_{E_i} Q)Z,X)g(QE_i,Y)+g((\nabla^{LC}_{E_i} Q)X,Y)g(QE_i,Z)\Big\}
\end{split}
\]

So after we take the difference we obtain
\[
\begin{split}
(2\delta_gS\wedge S-\delta_g(S\wedge S))(X,Y,Z) =&\, +2\sum_i\Big\{\,g((\nabla^{LC}_{E_i} Q)Y,Z)g(QE_i,X)+g((\nabla^{LC}_{E_i} Q)Z,X)g(QE_i,Y)\\
&\, +g((\nabla^{LC}_{E_i} Q)X,Y)g(QE_i,Z)\Big\}
\end{split}
\]

But
\[
\sum_ig((\nabla^{LC}_{E_i} Q)X,Y)g(QE_i,Z)=\sum_ig((\nabla^{LC}_{g(QE_i,Z)E_i} Q)X,Y) = -g(\nabla^{LC}_{QZ} Q)X,Y).
\]
\end{proof}

\begin{prop}\label{bracket}
Let $(M,g)$ be a Riemannian manifold and $Q:TM\to TM$ a skew-symmetric  endomorphism of the tangent bundle. Then
\begin{equation}\label{q}
\begin{split}
(\#_3[\tilde Q,\tilde Q]_s)(X,Y,Z)=&\, -2\Big\{QX(g(Y,QZ))+QY(g(Z,QX))+QZ(g(X,QY)) \\
& +g(X,[QY,QZ])+g(Y,[QZ,QX])+g(Z,[QX,QY])\Big\}
\end{split}
\end{equation}
\end{prop}

\begin{proof}
By Lemma \ref{lemma}, we need to compute $g((\nabla^{LC}_{QX}Q)Y,Z)$. Since $\nabla^{LC}g=0$ we obtain
\begin{align*}
g((\nabla^{LC}_{QX}Q)Y,Z)=&\, g(\nabla^{LC}_{QX}QY,Z)-g(Q(\nabla^{LC}_{QX}Y),Z) \\
=&\, g(\nabla^{LC}_{QX}Y,QZ)+g(\nabla^{LC}_{QX}QY,Z) \\
=&\, QX(g(Y,QZ))-g(Y,\nabla^{LC}_{QX}QZ)+g(\nabla^{LC}_{QX}QY,Z).
\end{align*}
Then
\begin{align*}
(\#_3[Q,Q]_s)(X,Y,Z) =&\,
  QX(g(Y,QZ))-g(Y,\nabla^{LC}_{QX}QZ)+g(\nabla^{LC}_{QX}QY,Z) \\
& QY(g(Z,QX))-g(Z,\nabla^{LC}_{QY}QX)+g(\nabla^{LC}_{QY}QZ,X) \\
& QZ(g(X,QY))-g(X,\nabla^{LC}_{QZ}QY)+g(\nabla^{LC}_{QZ}QX,Y).
\end{align*}
But $g(\nabla^{LC}_{QY}QZ,X)-g(X,\nabla^{LC}_{QZ}QY)=g(X,\nabla^{LC}_{QY}QZ-\nabla^{LC}_{QZ}QY)=g(X,[QY,QZ])$ since the Levi-Civita connection has no torsion, so we have the thesis.
\end{proof}

$\tilde Q$ is a twisted Poisson structure if there exists a closed 3-form $\phi$ such that
\[
[\tilde Q,\tilde Q]_s = \frac {1}{2} \Lambda^3 Q^{\#} \phi.
\]
Note that for any 1-form $\alpha$ we have $\tilde Q\alpha=Q \#^{-1}\alpha$, so
\[
(\Lambda^3 \tilde Q^{\#} \phi)(\alpha,\beta,\gamma)= \phi(Q\#^{-1}\alpha,Q\#^{-1}\beta,Q\#^{-1}\gamma).
\]
Then $\frac {1}{2} \Lambda^3 Q^{\#} \phi=\frac12\phi(QX,QY,QZ)$, so in view of \eqref{q} $\tilde Q$ is a twisted Poisson structure if and only if there exist a closed 3-form $\phi$ such that
\[
\phi(QX,QY,QZ)= -4 \sigma_{XYZ} \big(QX(g(Y,QZ))+g(X,[QY,QZ])\big).
\]

\begin{rem}\label{remInv}
If we suppose that $Q$ is invertible, i.e. it has maximal rank at any point, then $Q^{-1}$ is well defined, so $\tilde Q$ is non-degenerate and we can define its inverse $q(X,Y)=g(Q^{-1}X,Y)$. Moreover
\[
\begin{split}
\phi(X,Y,Z)=&\, -4\big\{ \sigma_{XYZ}(X(g(Q^{-1}Y,Z))) +\sigma_{XYZ}(g(Q^{-1}X,[Y,Z]))\big\} \\
=&\, -4\,dq(X,Y,Z),
\end{split}
\]
so $Q$ is a twisted Poisson structure. Note that it is Poisson if and only if $dq=0$, i.e. $q$ is a symplectic form. Moreover, if on a complex manifold $(M,J)$ we choose $Q=J$, then $q=\omega$ and we obtain
\begin{equation}\label{zabzine}
[\omega^{-1},\omega^{-1}]_s=\#^{-1}_3 (2Jd\omega).
\end{equation}
\end{rem}

The endomorphism $Q: TM \rightarrow TM$  in general has no constant rank and it defines two distributions (in the sense of \cite{Sussmann}) given by  $Im (Q)$ and $Ker (Q)$, that are orthogonal since $g(QX,Y)=-g(X,QY)$. We recall that   in general for a  distribution  $\mathcal D$ on M, also called a \lq \lq generalized distribution",   we mean a  mapping $\mathcal D$  which assigns to every  point  $p \in M$  a linear subspace $\mathcal D (p)$ of the tangent space $T_p M$  without requiring the dimension of $\mathcal D (p)$ to be constant.

A foliated manifold $M$ is one which can be expressed as a disjoint union of subsets called  leaves. A leaf is a connected submanifold (injective immersion) $L \subset M$  such that any point $p \in L$  has a neighborhood $U \subset M$  where the connected component of $P$  in  $L \cap U$  is an embedded submanifold of  $M$.  In a usual foliation
all the leaves have the same dimension, while a generalized foliation allows the dimension of the leaves to vary. In \cite{Sussmann} Sussmann  described necessary and sufficient conditions for a distribution $\mathcal D$ to be integrable into a generalized foliation (i.e. for every point $p\in M$ the tangent space to the leaf through $p$ coincides  with $\mathcal D (p)$).

A distribution $\mathcal D$ is involutive if for any two vector fields $X, Y \in \mathcal D$  their  commutator $[X, Y] \in {\mathcal D}$. If a distribution $\mathcal D$ is regular, i.e. if $\dim \mathcal D(p)$ is constant, then by the classical Frobenius Theorem the distribution  $\mathcal D$ is  involutive if and only if it is integrable. In general (i.e., when $\mathcal D$  has \lq\lq singularities\rq\rq) the condition that $\mathcal D$ is involutive is necessary but not sufficient for the integrability of $\mathcal D$.
However, as pointed out in \cite{Biraglia}, distributions locally of finite type satisfies a Frobenius-like result even if they are not regular in general. We say that a smooth distribution is \emph{locally of finite type} if for every $p \in M$ there are smooth vector fields $X_1, \ldots, X_n \in \mathcal D$ such that
\[
 {\mathcal D}(p) = span < X_1 (p), \ldots, X_n (p)>.
\]
By \cite[Theorem 8.1]{Sussmann} a smooth distribution ${\mathcal D}$ locally of finite type is integrable to a generalized foliation if and only if it is involutive.

Since $Q:TM \rightarrow TM$ is a bundle map of $TM$ and $Im (Q)$  is the
image of a smooth bundle map, then $Im (Q)$ is spanned by the smooth vector fields  $Y = Q(X)$ and hence it is a smooth distribution. Furthermore, for any point $p \in M$  we may choose a local basis of sections $X_1, \ldots X_n$  of  $TM$  in some neighbourhood $U$ of $p$. Then $Q(X_1),  \ldots, Q(X_n)$ certainly span  $Im (Q)$ on $U$, so $Im(Q)$ is locally of finite type and is integrable if and only if $[Im (Q),Im (Q)] \subset Im (Q)$. Moreover, note that the function $rank (Q)$  is  lower semi-continuous.

\begin{theorem}\label{propInt}
Let $(M,g)$ be a Riemannian manifold and $Q:TM\to TM$ a skew-symmetric  endomorphism of the tangent bundle of $M$.  If $Im (Q)$ is an involutive  distribution, then it is integrable to a generalized foliation, i.e. $M$  can be
expressed locally as a disjoint union of embedded submanifolds (the leaves) such that at any point
$p \in M$, the tangent space to the leaf through $p$ is precisely $Im(Q) (p)$ and the dimension of the leaf is a lower semi-continuous function on the manifold. Moreover, every leaf has  a twisted Poisson structure.
\end{theorem}

\begin{proof} The first part of the proposition follows by \cite[Theorem 8.1]{Sussmann} and by using that $Im(Q)$ is involutive, i.e. $[Im (Q),Im (Q)] \subset Im (Q)$. Moreover, since on every leaf $L$ the endomorphism $Q$ is non-degenerate, as noted in Remark \ref{remInv} $\tilde Q_L=\tilde Q\vert_L$ satisfy
\[
[\tilde Q_L,\tilde Q_L]_s = \frac {1}{2} \Lambda^3 Q_L^{\#} dq_L
\]
where $q_L(X,Y)=g(Q_L^{-1}X,Y)$.
\end{proof}

As a consequence we have the following.
\begin{cor}
Let $(M,g)$ be a Riemannian manifold and $Q:TM\to TM$ a skew-symmetric  endomorphism of the tangent bundle of $M$.  If $Im (Q)$ is an involutive  distribution of constant rank, then   $\tilde Q$ is a twisted Poisson structure \end{cor}
\begin{proof}  For every point $p\in M$ we have the orthogonal decomposition $T_pM=Im (Q) (p) \oplus Ker (Q) (p)$.  As a consequence of Theorem   \ref{propInt} the $3$-form $\phi$ $$
\phi(X,Y,Z)= \left\{
\begin{array}{ll}
(\#_3[\tilde Q,\tilde Q]_s)(Q^{-1}X,Q^{-1}Y,Q^{-1}Z) & \text{if}\ X,Y,Z\in Im(Q) \\
0 & \text{otherwise}
\end{array} \right.
$$
 is well defined. If one considers the 2-form
$$
q (X, Y) = \left \{ \begin{array}{ll}
g (Q^{-1} X, Y) & X, Y \in Im (Q),\\
0 & {\mbox{otherwise}},
\end{array} \right.
$$
then $\phi = - 4 d q$ and so $\phi$ is closed and
\[
\phi(QX,QY,QZ)= -4 \sigma_{XYZ} \big(QX(g(Y,QZ))+g(X,[QY,QZ])\big).
\]
 \end{proof}

\vspace{5mm}

\section{Twisted Poisson structures on Tamed almost Generalized K\"ahler structures}\label{sec:twistedGK}

Generalized complex and K\"ahler structures are deeply related with Poisson structures. If $\J$ is a generalized complex manifold on $M$, Gualtieri \cite{gual-ann} proved that $\beta=\pi_{TM}(\J|_{T^*M})$ is a Poisson structure. Moreover, let $(M, J_{\pm}, g, b)$ be a generalized K\"ahler manifold. Then the endomorphism of the tangent bundle $Q=[J_+,J_-]$ is skew-symmetric and satisfy $QJ_\pm=-J_\pm Q$, so the induced 2-form $S$ has type $(2, 0)+(0, 2)$ with respect to both $J_{\pm}$. In \cite{Hitchin} Hitchin proved that the $(2,0)$-part $\sigma_+$ (with respect to $J_+$) of the bivector $\tilde Q=\#^{-1}_2S$ is a holomorphic Poisson structure, i.e. $\debar\sigma_+=0$ and $[\sigma_+,\sigma_+]_s=0$. In particular, this implies that the bivector $\tilde Q=\#^{-1}_2S$ is a Poisson structure.

Now, let $(\J_\Omega,\J_2)$ be a tamed almost generalized K\"ahler structure on $M$ induced by $(\Omega,J)$. If $J_-=-\Omega^{-1}J^*\Omega$ is integrable, both $\J_\Omega$ and $\J_2$ induce a Poisson structure, which are given by
\[
 \beta^1=\omega^{-1}_+ + \omega^{-1}_-=2\Omega^{-1} \quad \text{and} \quad \beta^2=\omega^{-1}_+ - \omega^{-1}_-
\]
(note that the first one is simply the Poisson structure induced by the symplectic form $\Omega$). Moreover, the $(2,0)$-part of the bivector $\tilde Q$ induced by $Q=[J_+,J_-]$ is a holomorphic Poisson structure.

The goal of this section is to use the results obtained in section \ref{sec:twistedRiemann} to investigate the existence of twisted Poisson structures on a tamed almost generalized K\"ahler structure when $J_-$ is not integrable. First we consider the structures defined by $\pi_{TM}(\J|_{T^*M})$ on the generalized complex structures  $(\J_\Omega,\J_2)$.

\begin{lemma}\label{lemma:twistedGenCx}
Given a tamed almost generalized K\"ahler manifold $(M, J_\pm,g,b)$, then $\beta^1$ is a Poisson structure and
\begin{equation}\label{differenceomega}
[\beta^2,\beta^2]_s = 4\, \#_3^{-1}  ( J_+ d \omega_+ + J_- d \omega_-).
\end{equation}
In particular, $\beta^2$ is a Poisson structure if and only if $ J_- d\omega_- = - J_+d\omega_+$, i.e. if \eqref{eq:J+J-tamed} vanishes.
\end{lemma}

\begin{proof}  For any almost Hermitian manifold $(M, J, g)$ the associated fundamental $2$-form $\omega$ satisfy the relation \eqref{zabzine}, as also proved in \cite{LZ}. Applying this result, we obtain
\[
[\omega_+^{-1},\omega_+^{-1}]_s=\#^{-1}_3 (2J_+d\omega_+), \qquad [\omega_-^{-1},\omega_-^{-1}]_s=\#^{-1}_3 (2J_-d\omega_-).
\]
Moreover, by using that $\beta^1 = 2\Omega^{-1}$ is a Poisson structure, we find that
\[
[\omega^{-1}_+ , \omega^{-1}_-]_s = - \#_3^{-1}  ( J_+ d \omega_+ +  J_- d \omega_-),
\]
and then \eqref{differenceomega}.
\end{proof}

\begin{prop}
Let $(\J_\Omega,\J_2)$ be a tamed almost generalized K\"ahler structure on $M$ induced by  a Hermitian symplectic structure $(\Omega,J)$. Then the bivector $\beta^2$ is a twisted Poisson structure if and only if there is a closed 3-form $\phi$ such that $J_- d\omega_-+ J_+d\omega_+ = \frac18(J_+-J_-)\phi$.
\end{prop}
\begin{proof}
Note that $\beta^2\alpha=(J_+-J_-)\#^{-1}\alpha$ for any $\alpha\in T^*M$, so the twisted Poisson condition becomes
\[
 [\beta^2,\beta^2]_s = \frac12\#_3^{-1} (J_+-J_-)\phi
\]
for some closed 3-form $\phi$. Then the proof follows by Lemma \ref{lemma:twistedGenCx}.
\end{proof}

Now, let us consider the endomorphism $Q=[J_+,J_-]:TM\to TM$. In the integrable case, the $(2,0)$-part of the bivector $\tilde Q$ induced by $Q$ as in Section \ref{sec:twistedRiemann} is a holomorphic Poisson structure. However, when $J_-$ is not integrable, we can apply Theorem \ref{propInt} to obtain conditions for the existence of a \lq\lq local \rq\rq twisted Poisson structure.
\begin{prop}\label{prop:twistFoliation}
Let $(\J_\Omega,\J_2)$ be a tamed almost generalized K\"ahler structure on $M$ induced by  a Hermitian symplectic structure $(\Omega,J)$. If $Im(Q)$ is an involutive  distribution, then it is integrable to a generalized foliation and on every leaf $L$ the bivectors $\tilde Q_L=\tilde Q\vert_L$ and $\beta^2_L=\beta^2\vert_L,$ are twisted Poisson structures.
\end{prop}

\begin{proof}
 The only non-trivial part is to prove that $\beta^2_L$ is a twisted Poisson structure on every leaf of the foliation induced by $Im(Q)$. Since $\beta^2\alpha=(J_+-J_-)\#^{-1}\alpha$ for any $\alpha\in T^*M$, $\beta^2$ is the bivector induced as in Section \ref{sec:twistedRiemann} by the endomorphism $J_+-J_-$. Moreover, since $Q=(J_+-J_-)(J_++J_-)$ and $J_++J_-$ is invertible, $J_+-J_-$ is non-degenerate if and only if $Q$ is non-degenerate. Then on every leaf $L$ the restriction $(J_+-J_-)\vert_L$ is non-degenerate, so as in the proof of Theorem \ref{propInt}
\[
 [\beta^2_L,\beta^2_L]_s= \frac12 \Lambda^3 (J_*+J_-)\vert_L^\# dp_L,
\]
where $p_L(X,Y)=g\vert_L\big((J_+-J_-)\vert_L^{-1}X,Y\big)$ for any vector fields $X,Y$ on $L$.
\end{proof}

If $J_-$ is integrable, in \cite{IKR} it has been shown that the spaces orthogonal to the kernels of $(J_+ \pm J_-)$ and of the commutator $[J_+, J_-]$ (that coincide with the images of these operators since they are skew-symmetric) are always integrable as generalized distributions and thus they are in particular involutive. As remarked in \cite{LZ} this property follows also from the fact that the spaces orthogonal to the kernels of $(J_+ \pm J_-)$ correspond to the symplectic leaves of  $\omega^{-1}_+ + \omega^{-1}_-$  and $\omega^{-1}_+ - \omega^{-1}_-$ respectively.

In the following Proposition we compute $[\tilde Q,\tilde Q]_s$ for a tamed almost generalized K\"ahler structure.
\begin{prop}\label{prop:BracExplicit}
Let $(\J_\Omega,\J_2)$ be a tamed almost generalized K\"ahler structure on $M$ induced by  a Hermitian symplectic structure $(\Omega,J)$. Then
\begin{align}\label{eq:BracExplicit}
\#_3[\tilde Q,\tilde Q]_s (X,Y,Z) =&\, -\sigma_{X,Y,Z} \Big[ (N_-+3\b N_-)(QX,J_-Y,J_+Z)+(N_-+3\b N_-)(QX,J_+Y,J_-Z) \\
&  -J_+d\omega_+(QX,J_-Y,J_+Z)-J_+d\omega_+(QX,J_+Y,J_-Z) -J_+d\omega_+(QX,Y,J_+J_-Z) \\
&  -J_+d\omega_+(QX,J_+J_-Y,Z) + J_-d\omega_-(QX,J_-Y,J_+Z) + J_-d\omega_-(QX,Y,J_-J_+Z)\\
&  + J_-d\omega_-(QX,J_-J_+Y,Z) + J_-d\omega_-(QX,J_+Y,J_-Z) \Big],
\end{align}
where $N_-(X,Y,Z)=g(X,N_-(Y,Z))$ is the 3-tensor associated to the Nijenhuis tensor of $J_-$ and $\b N_-$ is the skew-symmetric part of $N_-$ defined by
\[
 (\b N_-) (X, Y, Z) = \frac13 [ N_-(X, Y, Z)  + N_-(Y, Z, X) + N_- (Z, X, Y)].
\]
\end{prop}

\begin{proof}
By Lemma \ref{lemma}, we need to compute $g((\nabla^{LC}_{QX}Q)Y,Z)$. In the proof we use the Bismut connections $\nabla^\pm$ of the (almost) complex structures $J_\pm$, which satisfy $\nabla^\pm J_\pm = 0$ and following \cite{Gauduchon} are defined by
\[
\begin{split}
& g(\nabla^+_XY,Z) = g(\nabla^{LC}_XY,Z)-\frac12 J_+d\omega_+(X,Y,Z) \\
& g(\nabla^-_XY,Z) = g(\nabla^{LC}_XY,Z)-\frac14(N_-+3\b N_-)(X,Y,Z)-\frac12 J_-d\omega_-(X,Y,Z).
\end{split}
\]

Using these definitions,
\[
 \begin{split}
  g((&\nabla^{LC}_{QX}Q)Y,Z) = \\
  = &\, -g(\nabla^{LC}_{QX}Y,QZ)-g(\nabla^{LC}_{QX}QY,Z) \\
  = &\, -g(\nabla^+_{QX}Y,QZ)-g(\nabla^+_{QX}QY,Z)-\frac12 J_+d\omega_+(QX,Y,QZ)-\frac12 J_+d\omega_+(QX,QY,Z) \\
  = &\, -g((\nabla^+_{QX}J_-)Y,J_+Z)-g((\nabla^+_{QX}J_-)J_+Y,Z)-\frac12 J_+d\omega_+(QX,Y,QZ)-\frac12 J_+d\omega_+(QX,QY,Z).
 \end{split}
\]
By definition of $\nabla^+$ and $\nabla^-$ we obtain that
\[
\begin{split}
 g((\nabla^+J_-)Y,Z)= &\, -\frac12 (N_-+3\b N_-)(X,J_-Y,Z)+\frac12(J_+d\omega_+-J_-d\omega_-)(X,Y,J_-Z)\\
  &\, +\frac12(J_+d\omega_+-J_-d\omega_-)(X,J_-Y,Z).
\end{split}
\]
Using this condition in the previous equation, we obtain the thesis.
\end{proof}

\begin{rem}
  We know that when $J_-$ is integrable, $\tilde Q$ is a Poisson structure, i.e. $[\tilde Q,\tilde Q]_s=0$. Indeed, if $J_-$ is integrable $N_-=0$ and $J_-d\omega_-=-J_+d\omega_+$, so $J_+d\omega_+$ has type $(2,1)+(1,2)$ with respect to both complex structures, then \eqref{eq:BracExplicit} becomes
\begin{align*}
\#_3[\tilde Q,\tilde Q]_s (X,Y,Z) =\sigma_{X,Y,Z} \Big[ &\, 2J_+d\omega_+(QX,J_+Y,J_-Z) +2J_+d\omega_+(QX,J_-Y,J_+Z) -  \\
&\,  +J_+d\omega_+(QX,PY,Z) +J_+d\omega_+(QX,Y,PZ) \Big],
\end{align*}
where $P=J_+J_-+J_-J_+$, and arguing as in \cite{DGMY} we obtain $[\tilde Q,\tilde Q]_s=0$.
\end{rem}

In \cite{Hitchin}  it was shown that for a generalized K\"ahler structure $(g,J_+,J_-, b)$ the skew form  $g(Q X, Y )$  for the bihermitian metric is of type $(2, 0)+(0, 2)$ and defines a holomorphic Poisson structure for either complex structure $J_+$ or $J_-$. For a  tamed almost generalized K\"ahler structure we can prove the following theorem.

\begin{theorem}\label{condhol(2,0)} Let $(\J_\Omega,\J_2)$ be a tamed almost generalized K\"ahler structure on $M$ induced by  a Hermitian symplectic structure $(\Omega,J)$. Then the $(2,0)$-part  $\tilde{Q}^{(2,0)}$ of the bivector  $\tilde Q$  is holomorphic if and only if  $\psi^{(3,0)} = 0$, where
$$\psi (X, Y, Z) = \Omega(X,N_-(Y,Z)).$$
\end{theorem}

\begin{proof} Let $D^+$  be the Chern connection of $(g,J_+)$. Then $$g(D^{+}_XY,Z) = g(\nabla^{LC}_XY,Z) + \frac{1}{2}d\omega_{+}(J_{+}X,Y,Z),$$ where $\omega_+ (X, Y) = g (X, J_+ Y)$ is the fundamental 2-form of $(J_+, g)$. From  \eqref{eq:J+J-tamed}  we have:
$$
J_+ d\omega_+= db, \quad   (J_+d\omega_+  + J_-d\omega_-) (X,Y,Z) =  - \frac12 \sigma_{XYZ}[\Omega(J_+X,N_-(Y,Z))].$$
By \cite[Proposition 1]{Gauduchon}   we have  the following formula for the covariant derivative of $J_-$:
 $$
\begin{array}{c}
2g((\nabla^{LC}_XJ_-)(Y),Z)= - d\omega_-(X,Y,Z) + d\omega_-(X,J_-Y,J_-Z)  - g(J_-X,N_-(Y,Z))= \\  - (d\omega_-)^+(X,Y,Z) + (d\omega_-)^+(X,J_-Y,J_-Z)\\
 - \frac 12 (g(J_-X,N_-(Y,Z)) + g(J_-Y,, N_-(X,Z)) -g(J_-Z,N_-(X,Y)))
 \end{array}
 $$
where $(d\omega_-)^+$ denotes the $(2,1)+(1,2)$-component of the 3-form $d\omega_-$ with respect to $J_-$. In particular $$(d\omega_-)^+(X,Y,Z)-(d\omega_-)^+(X,J_-Y,J_-Z) = (d\omega_-)^+(J_-X,Y,J_-Z)+(d\omega_-)^+(J_-X,J_-Y,Z)$$ (see \cite[formula (1.3.6)]{Gauduchon}). Similarly from $(d\omega_-)^{(2,1)}=0$ follows $$(d\omega_-)^-(X-iJ_-X,Y+iJ_-Y,Z+iJ_-Z)=0$$ and $$(d\omega_-)^-(X,Y,Z)-(d\omega_-)^-(X,J_-Y,J_-Z) = -(d\omega_-)^-(J_-X,Y,J_-Z)-(d\omega_-)^-(J_-X,J_-Y,Z),$$
where by $(d\omega_-)^-$  we denote the $(3,0)+(0,3)$-component of the 3-form $d\omega_-$ with respect to $J_-$.
 From above
$$
\begin{array}{c}
2g((\nabla^{LC}_XJ_-)(Y),Z)= - (d\omega_-)^+(X,Y,Z)  + (d\omega_-)^+(X,J_-Y,J_-Z)\\
- (d\omega_-)^-(X,Y,Z)  +  (d\omega_-)^-(X,J_-Y,J_-Z) - g(J_-X,N_-(Y,Z))\\
=- (d\omega_-)^+(J_-X,Y,J_-Z)  - (d\omega_-)^+(J_-X,J_-Y,Z) \\
+  (d\omega_-)^-(J_-X,Y,J_-Z)  + (d\omega_-)^-(J_-X,J_-Y,Z)  -  g(J_-X,N_-(Y,Z))\\
= - d\omega_-(J_-X,Y,J_-Z) - d\omega_-(J_-X,J_-Y,Z) \\
+ 2(d\omega_-)^-(J_-X,Y,J_-Z) + �2(d\omega_-)^-(J_-X,J_-Y,Z)  - g(J_-X,N_-(Y,Z))\\
=- d\omega_-(J_-X,Y,J_-Z) - d\omega_-(J_-X,J_-Y,Z) \\
- 2(d\omega_-)^-(X,Y,Z) + 2(d\omega_-)^-(X,J_-Y,J_-Z)  - g(J_-X,N_-(Y,Z))\\
=  -  d\omega_-(J_-X,Y,J_-Z) -  d\omega_-(J_-X,J_-Y,Z) \\
+2\phi_1  + g(J_-X,N_-(Y,Z))
\end{array}
$$
  where $$ \begin{array}{lcl} \phi_1(X,Y,Z) &=& - (d\omega_-)^-(X,Y,Z) + (d\omega_-)^-(X,J_-Y,J_-Z)  - g(J_-X,  N_- (Y,Z))\\[4 pt]
 & = &  - \frac 12 (g(J_-X,N_-(Y,Z))- \frac 12 g(J_-Y,, N_-(X,Z)) + \frac 12 g(J_-Z,N_-(X,Y))).
 \end{array}$$
  Denote by $$\begin{array}{ll} \psi_1(X,Y,Z)= - \frac12 \sigma_{XYZ}[\Omega(J_+X,N_-(Y,Z))],  & \phi_2 (X, Y, Z) = 2\phi_1  (X,Y, Z) + g(J_-X,N_-(Y,Z)),\\[4 pt]
\psi_2(X,Y,Z) = \psi_1(X,J_-Y,Z)+\psi_1(X, Y J_-Z), & \psi = \phi_2+\psi_2.
\end{array}
$$
Then from \cite{DGMY}:
$$
\begin{array}{lll}
2g((D^+_XJ_-)(Y),Z)=2g(D_X^+J_-Y,Z)+2g(D^+_XY,J_-Z)=\\[6pt]
2g((\nabla^{LC}_XJ_-)(Y),Z) + d\omega_+(J_+X,J_-Y,Z) + d\omega_+(J_+X,Y,J_-Z)=\\[6pt]
- d\omega_-(J_-X,Y,J_-Z) - d\omega_-(J_-X,J_-Y,Z)+ \phi_2(X,Y,Z)\\
\hspace{3.5cm} + d\omega_+(J_+X,J_-Y,Z) + d\omega_+(J_+X,Y,J_-Z)=\\[6pt]
- d^-\omega_-(X,J_-Y,Z) - d^-\omega_-(X,Y,J_-Z)\\[6pt]
\hspace{1.5cm} - d^+\omega_+(X,J_+J_-Y,J_+Z) - d^+\omega_+(X,J_+Y,J_+J_-Z)+ \phi_2(X,Y,Z)
\end{array}
$$
and since $d^{\pm}\omega_{\pm}(X,Y,Z)=-d\omega_{\pm}(J_{\pm}X,J_{\pm}Y,J_{\pm}Z) = -  J_{\pm}d\omega_{\pm} (X, Y, Z)$  the equation (6) from \cite{DGMY} becomes:
\begin{equation}\label{Chern1}
\begin{array}{lll}
2g((D^+_XJ_-)(Y),Z)= + d^+\omega_+(X,J_-Y,Z) + d^+\omega_+(X,Y,J_-Z)\\[6pt]
\hspace{3.5cm} - d^+\omega_+(X,J_+J_-Y,J_+Z) - d^+\omega_+(X,J_+Y,J_+J_-Z)\\[6pt]
\hspace{3.5cm} + \psi(X,Y,Z)
\end{array}
\end{equation}
Because $J_+$ is integrable again we have:
$$
\begin{array}{lll}
d^+\omega_+(A,B,C)=\\[6pt]
\hspace{1.5cm}d^+\omega_+(J_+A,J_+B,C)+d^+\omega_+(J_+A,B,J_+C)+d^+\omega_+(A,J_+B,J_+C)
\end{array}
$$
So (7) from \cite{DGMY} becomes:
\begin{equation}\label{Chern2}
\begin{array}{lll}
2g((D^+_XJ_-)(Y),Z)= +  d^+\omega_+(J_+X,J_-Y,J_+Z) + d^+\omega_+(J_+X,J_+J_-Y,Z)\\[6pt]
\hspace{2.5cm} + d^+\omega_+(J_+X,Y,J_+J_-Z) + d^+\omega_+(J_+X,J_+Y,J_-Z)+\psi(X,Y,Z)
\end{array}
\end{equation}
Thus,  for $Q=[J_+,J_-]$ we get
$$
\begin{array}{c}
2g((D^+_{X}Q)(Y),Z)-2g((D^+_{J_+X}Q)(Y),J_+Z)=\\[6pt]
-2g((D^+_{X}J_-)(Y),J_+Z)-2g((D^+_{X}J_-)(J_+Y),Z)\\[6pt]
-2g((D^+_{J_+X}J_-)(Y),Z)+2g((D^+_{J_+X}J_-)(J_+Y),J_+Z)
\end{array}
$$
 Then, after applying (\ref{Chern1}) to the first and the second term, and
(\ref{Chern2}) to the third and the fourth term, we easily get
 the last and most important identity (8) from  \cite{DGMY} in the form of:
$$
\begin{array}{lll}\label{derP}
2g((D^+_{X}Q)(Y),Z)-2g((D^+_{J_+X}Q)(Y),J_+Z)&=&\psi(J_+X,J_+Y,J_+Z)-\psi(J_+X,Y,Z)-\psi(X,J_+Y,Z)\\
& &-\psi(X,Y,J_+Z)
\end{array}
$$
From here we notice that
$$
\begin{array}{lll}
\psi(X-iJ_+X,Y-iJ_+Y,Z-iJ_+Z)&=& \psi(X,Y,Z)-\psi(J_+X,J_+Y,Z)\\
& &-\psi(J_+X,Y,J_+Z)-\psi(X,J_+Y,J_+Z)\\
 & &+ i(\psi(J_+X,J_+Y,J_+Z)-\psi(J_+X,Y,Z)\\
 & &-\psi(X,J_+Y,Z)-\psi(X,Y,J_+Z))
\end{array}
$$
Then further we proceed as similarly to obtain:
$$
\begin{array}{lrr}
g(D^{+}_{X+iJ_+X}\widetilde{Q}^{(2,0)},Y\wedge Z)=
2(D^{+}_{X+iJ_+X}S^{(0,2)})(Y,Z)=\\[6pt]
[g((D^{+}_{X}Q)(Y),Z)+ig((D^{+}_{X}Q)(Y),J_+Z)]\\[6pt]
\hspace{0.3cm}+i[g((D^{+}_{J_+X}Q)(Y),Z)+ig((D^{+}_{J_+X}Q)(Y),J_+Z)]=\\[6pt]
[g((D^{+}_{X}Q)(Y),Z)- g((D^{+}_{J_+X}Q)(Y),J_+Z)]\\[6pt]
\hspace{0.3cm}+i[g((D^{+}_{X}Q)(Y),J_+Z)+g((D^{+}_{J_+X}Q)(Y),Z)],
\end{array}
$$
which leads to
\begin{equation}
2(D^{+}_{X+iJ_+X}S^{(0,2)})(Y,Z)= \psi(X-iJ_+X,Y-iJ_+Y,Z-iJ_+Z) = \psi^{(0,3)}(X,Y,Z).
\end{equation}
So $\tilde{Q}^{(2,0)}$ is holomorphic iff $\psi^{(3,0)} = 0$.

Now we compute $\psi$. By recalling that $g=\frac12 \Omega(J_++J_-)$, $g(J_- X,N_-(Y,Z))=g(X,N_-(J_-Y,Z))=g(X,N_-(Y,J_- Z))$ and $\Omega(J_+X,Y)=-\Omega(X,J_-Y)$ we obtain that
\[
\begin{split}
  \phi_2 (X,Y,Z) =&\, 2\phi_1(X,Y,Z)+g(J_-X,N_-Y,Z) = g(J_-Y,N_-(Z,X))+g(J_-Z,N_-(X,Y))\\
  &\, \frac12\Big[ \Omega(J_+J_-Y,N_-(Z,X)) + \Omega(J_+J_-Z,N_-(X,Y)) - \Omega(Y,N_-(Z,X)) - \Omega(Z,N_-(X,Y)) \Big].
\end{split}
\]
Moreover,
\[
\begin{split}
  \psi_2(X,Y,Z)=&\, \psi_1(X,J_-Y,Z)+\psi_1(X,Y,J_-Z) \\
  =&\, \frac12\Big[ -\Omega(J_+X,N_-(J_-Y,Z))-\Omega(J_+J_-Y,N_-(Z,X))-\Omega(J_+Z,N_-(X,J_-Y))\\
  &\,  -\Omega(J_+X,N_-(Y,J_-Z)) -\Omega(J_+Y,N_-(J_-Z,X))-\Omega(J_+J_-Z,N_-(X,Y)) \Big]\\
  =&\, \frac12\Big[ 2\Omega(X,N_-(Y,Z))-\Omega(J_+J_-Y,N_-(Z,X))+\Omega(Z,N_-(X,Y))\\
  &\,  +\Omega(Y,N_-(Z,X))-\Omega(J_+J_-Z,N_-(X,Y)) \Big]\\
\end{split}
\]
so
\[
  \psi(X,Y,Z) = \phi_2 (X,Y,Z)+\psi_2 (X,Y,Z) = \Omega(X,N_-(Y,Z)).
\]
 \end{proof}

\begin{rem} The $(3,0)$-part of $\psi$ can be written as
\[
\begin{split}
  \psi^{(3,0)} (X, Y, Z) =&\, \psi(X,Y,Z)-\psi(J_+X,J_+Y,Z) -\psi(J_+X,Y,J_+Z)-\psi(X,J_+Y,J_+Z)\\
 &\, + i(\psi(J_+X,J_+Y,J_+Z)-\psi(J_+X,Y,Z) -\psi(X,J_+Y,Z)-\psi(X,Y,J_+Z)) \\
 =&\, \Omega(X,N_-(Y,Z)+J_-N_-(J_+Y,Z)+J_-N_-(Y,J_+Z)-N_-(J_+Y,J_+Z)) \\
 &\, -i\Omega(J_+X,N_-(Y,Z)+J_-N_-(J_+Y,Z)+J_-N_-(Y,J_+Z)-N_-(J_+Y,J_+Z))
 \end{split}
\]
thus $\psi^{(3,0)}=0$ if and only if the tensor
\[
  \mathfrak N(Y,Z)= N_-(Y,Z)+J_-N_-(J_+Y,Z)+J_-N_-(Y,J_+Z)-N_-(J_+Y,J_+Z)
\]
vanishes. \end{rem}

\begin{ex} {\bf (A compact $6$-dimensional example)}
Now we consider an example in dimension greater than 4. Let $\mathfrak g$ be the unimodular solvable six-dimensional Lie algebra with structure equations
\[
\begin{cases}
de^1=e^{26} \\
de^2=-e^{16} \\
de^3=e^{46} \\
de^4=-e^{36} \\
de^5=de^6=0
\end{cases}
\]
together with the complex structure $J_+$ defined by $J_+ e_1=-e_2,\ J_+ e_3=-e_4,\ J_+ e_5=-e_6$ and the symplectic form $\Omega=e^{12}+e^{34}+e^{56}+e^{16}$. Then ($J_+,\Omega$) is a Hermitian symplectic structure, but the almost complex structure $J_-=-\Omega^{-1} J_+^* \Omega$ defined by
\[
{J_-e_1=-e_6 \quad J_-e_2=e_1-e_5 \quad J_-e_3=-e_4 \quad J_-e_4=e_3 \quad J_-e_5=e_2-e_6 \quad J_-e_6=e_1}
\]
is not integrable. Moreover, the almost bi-Hermitian metric is
$$
g = \frac{1}{2}(\Omega J_+ -J_+^*\Omega) = \left (
\begin{array}{rrrrrr}
1&0&0&0&-\frac12&0\\
0&1&0&0&0&-\frac12\\
0&0&1&0&0&0\\
0&0&0&1&0&0\\
-\frac12&0&0&0&1&0\\
0&-\frac12&0&0&0&1
\end{array}
\right )
 $$
with fundamental 2-forms
\[
\omega_+ = e^{12}+e^{34}+e^{56}+\frac12(e^{16}-e^{25}) \qquad {\omega_- = e^{34}+\frac 12(e^{12}+e^{56}+e^{25})},
\]
and
\[
J_+d\omega_+ = \frac12\, e^{256},\qquad J_-d\omega_- = -\frac12\, e^{126}.
\]
Evaluating $\tilde Q=\frac 12 [J_+,J_-]g^{-1}=(J_+-J_-)\Omega^{-1}$ we obtain that
$$
\tilde Qe^1=-e_5, \quad \tilde Qe^2=e_6, \quad \tilde Qe^3=0, \quad \tilde Qe^4=0, \quad \tilde Qe^5= e_1, \quad \tilde Qe^6=-e_2,
$$
considering $\tilde Q$ as a morphism $T^* M \to TM$.
$\tilde Q$ is skew-symmetric, so we can see it as $\tilde Q=e_{51}+e_{26}\in\Lambda^2TM$; moreover $\tilde Q(J\alpha,J\beta)=-\tilde Q(\alpha,\beta)$, so $\tilde Q\in\Lambda^2T^{1,0}\oplus \Lambda^2T^{0,1}$.
We want to know if the bivector $\tilde Q$ is a Poisson structure, i.e. $[\tilde Q,\tilde Q]_s=0$. We compute
$$
[\tilde Q,\tilde Q]_s=2\,e_{126}
$$
so $Q$ is not a Poisson structure. However, $\tilde Q$ is a twisted Poisson structure, indeed
\[
e_{126} = -\Lambda^3 \tilde Q^\# (e^{256}),
\]
and $de^{256}=0$. Indeed, since $Q$ is given by
\[
Qe_1=-e_1+2e_5 \quad Qe_2=e_2-2e_6 \quad Qe_3=0 \quad Qe_4=0 \quad Qe_5=-2e_1+e_5 \quad Qe_6=2e_2-e_6,
\]
$Im(Q)=\langle e_1,e_2,e_5,e_6\rangle$, and $[X,Y]\in \langle e_1,e_2\rangle\subset Im(Q)$ for any $X,Y\in Im(Q)$.

To conclude the example, we prove that the $(2,0)$-part of $\tilde Q$ is not holomorphic. Indeed, by computing $\psi$ we obtain
\[
  \psi = e^5\otimes (e^{12}-e^{25}+e^{56})+e^6\otimes (e^{26}-e^{15})
\]
so the $(3,0)$-part of $\psi$ becomes
\[
  \psi^{3,0}= 3 (e^6+ie^5)\otimes(e^{26}-e^{15}+ie^{25}+ie^{16}),
\]
that is not zero, so  the $(2,0)$-part  $\tilde{Q}^{(2,0)}$ of the bivector  $\tilde Q$  is   not  holomorphic.

\end{ex}

\vspace{5mm}

\subsection{Complex surfaces}
Here we consider the case of 4-dimensional manifolds. When $(M,J)$ is a complex surface we have some additional properties. For instance, we have no $(3,0)$ or $(0,3)$-forms, and if $M$ is compact, we know that if $\Omega$ is a symplectic form taming $J$ then $J$ admits a K\"ahler metric \cite{LiZhang}. Thanks to these peculiarities, we prove stronger versions of Propositions \ref{prop:twistFoliation} and \ref{prop:BracExplicit}.

\begin{prop}
Let $(M,J)$ be a complex surface and $\Omega$ a symplectic form that tames $J$. Then the distribution $ImQ$ induced by $Q=[J_+,J_-]$ is involutive.
\end{prop}
\begin{proof}
Note that $rank(ImQ_p)=rank(\tilde Q_p)$ for every $p\in M$, and $\tilde Q$ is a $(2,0)+(0,2)$ bivector. Since $\dim M=4$, the space of $(2,0)+(0,2)$ bivectors is locally generated by
\[
\left( \frac{\partial}{\partial x^1}\wedge \frac{\partial}{\partial x^3}-\frac{\partial}{\partial x^2}\wedge \frac{\partial}{\partial x^2},\frac{\partial}{\partial x^1}\wedge \frac{\partial}{\partial x^4}+\frac{\partial}{\partial x^2}\wedge \frac{\partial}{\partial x^3} \right),
\]
where $(\frac{\partial}{\partial x^i})$ is a local frame of $TM$ such that $J\frac{\partial}{\partial x^1}=\frac{\partial}{\partial x^2}$ and $J\frac{\partial}{\partial x^3}=\frac{\partial}{\partial x^4}$. So for every point either $\tilde Q$ is zero or it is nondegenerate, i.e. $rank(\tilde Q)=4$, and the same holds for $ImQ$.
\end{proof}
Thus on a 4-dimensional manifold $ImQ_p$ either coincide with $T_pM$ or it is $\{ 0 \}$. Since $ImQ$ is a distribution locally of finite type, the rank is a lower semicontinous function, i.e. for every point we have an open neighbourhood where the rank cannot decrease. Moreover, by Proposition \ref{propInt} $ImQ$ is integrable, so the leafs of the foliation are either
\begin{itemize}
 \item $M$ except some points $p_i$;

 or

 \item the points $p_i$.
\end{itemize}

\begin{prop} Let $(M,J)$ be a complex surface and $\Omega$ a symplectic form that tames $J$, then
\begin{align}\label{eq:bracDim4}
\notag \#^{-1}_3[\tilde Q&,\tilde Q](X,Y,Z)=  \\
 =&\Big(\sigma_{XYZ} \big(-N_- (QX, J_- Y, J_+ Z) - N_- (QX, J_+ Y, J_- Z) - (J_+ d \omega_+ + J_- d \omega_-) (QX, Y, QZ) \big)\Big).
\end{align}
\end{prop}
\begin{proof} In dimension $4$ we have that any $3$-form is of type $(1,2) + (2,1)$ with respect to $J_+$ and $J_-$, thus
\[
\sigma_{XYZ} J_+ d \omega_+ (A, B, C) = \sigma_{X, Y, Z} J_+ d \omega_+ (J_+ A,  J_+ B, C) = \sigma_{XYZ} J_+ d \omega_+ (J_- A,  J_- B, C).
\]
and $\b N_-=0$. Therefore, by a direct computation we obtain
\[
\begin{array}{l}
\sigma_{XYZ} [-J_+d\omega_+(QX,J_-Y,J_+Z) - J_+d\omega_+(QX,J_+Y,J_-Z)  - J_+d\omega_+(QX,J_+ J_- Y, Z)\\
 - J_+d\omega_+(QX, Y, J_+ J_- Z) = \sigma_{XYZ}J_+d\omega_+ (Q X, Y, QZ).
\end{array}
\]
Similarly
\[
\begin{array}{l}
\sigma_{XYZ} [J_-d\omega_-(QX,J_-Y,J_+Z) + J_-d\omega_-(QX,J_+Y,J_-Z)  + J_-d\omega_-(QX,J_- J_+Y, Z) +\\
 J_-d\omega_-(QX, Y, J_- J_+ Z) = -\sigma_{XYZ}J_-d\omega_- (Q X, QY, Z) =  \sigma_{XYZ}J_-d\omega_- (Q X, Y, QZ),
\end{array}
\]
and \eqref{eq:bracDim4} follows.
\end{proof}

Finally, as a consequence of Theorem \ref{condhol(2,0)} we have the following
\begin{cor}  Let $(M,J)$ be a complex surface and $\Omega$ a symplectic form that tames $J$, then the $(2,0)$-part  $\tilde{Q}^{(2,0)}$ of the bivector  $\tilde Q$  is holomorphic. \end{cor}

\end{document}